\documentclass[a4paper]{scrartcl}

\usepackage[english]{babel}
\usepackage{amsfonts,amsthm,amssymb,amsmath}
\usepackage{mathtools}
\usepackage{letterswitharrows}
\usepackage[colorlinks]{hyperref}
\usepackage[capitalize,english]{cleveref}
\usepackage{enumitem}

\newcommand{\N}{\mathbb{N}}

\newcommand{\R}{\mathbb{R}}
\newcommand{\cF}{\mathcal{F}}
\newcommand{\cT}{\mathcal{T}}
\renewcommand\phi{\varphi}
\newcommand{\abs}[1]{|#1|}
\newcommand{\menge}[1]{\left\{#1\right\}}
\newcommand{\braces}[1]{\left(#1\right)}
\newcommand{\sub}{\subseteq}
\newcommand{\sm}{\smallsetminus}
\newcommand{\join}{\vee}

\newtheorem{THM}{Theorem}
\newtheorem{LEM}[THM]{Lemma}
\title{A Note on Generic Tangle Algorithms}
\author{Christian Elbracht\and Jakob Kneip\and Maximilian Teegen}

\begin{document}
\maketitle
\begin{abstract}
In this note we gather the theoretical outlines of three basic algorithms for tangles in abstract separation systems: a naive tree search for finding tangles; an algorithm which outputs a certificate for the non-existence of tangles if possible, and otherwise a way to jump-start the naive tree search; and a way to obtain a tree-of-tangles.
\end{abstract}

The algorithms we will describe are formulated in the setup and the language of \cite{ASS}.
In comparison to the algorithms of \cite{GroheSchweitzerTangleAlg} these are laid out for a more general setup, whereas the algorithms of \cite{GroheSchweitzerTangleAlg} are more elaborate.

The algorithms, in particular the first and the third, are what we have been using as general-purpose tools for our own explorative research in Hamburg over the past two years.
These algorithms are derived from proofs of the corresponding theorems from tangle theory, and we make no attempts to optimise their runtimes.

\section{Na\"ive tangle search algorithm}

Given a separation system $(\vS,\le,{}^*)$ and a set $\cF\sub 2^\vS$ we want to find all $\cF$-tangles of~$S$. Typically we will not be given the entire set $\cF$, but instead have oracle access to $\cF$: for any subset $\sigma$ of $\vS$ we will be able to check whether $\sigma$ lies in $\cF$. Note that we do not assume or use here that the elements of $\cF$ are stars.

In some cases the separation system $S$ comes with an order function, and we might want to compute the set of $k$-tangles of $S$ for different values of~$k$. In other cases, for instance if $S$ are the questions of some questionnaire, we might have some enumeration of $S$ and want to know the set of $\cF$-tangles of the first $k$ elements of $S$ for every $k$. Our algorithm described below will solve the latter problem. This includes the first problem as a special case: by enumerating the elements of $S$ in increasing order the $\cF$-tangles computed by the algorithm will include the set of $k$-tangles of $S$ for every $k$.

\paragraph{Input}
The separation system $S$, enumerated as $s_1,\dots,s_n$; oracle access to $\cF$.

\paragraph{Output}
For every $i\le n$ a list $\cT_i$ of all $\cF$-tangles of $\{s_1,\dots,s_i\}$.

\paragraph{Algorithm}
Set $\cT_0\coloneqq\{\emptyset\}$.

For $i$ from $1$ to $n$ perform the following:

Initialise $\cT_i$ as $\emptyset$. For every $\tau\in\cT_{i-1}$ and for both orientations of $s_i$ check whether that orientation $\vs_i$ of $s_i$ can be added to $\tau$, that is, whether $\tau$ contains a separation that points away from $\vs_i$ or a subset which, together with $\vs_i$, lies in $\cF$. If $\vs_i$ can be added to $\tau$ add $\tau\cup\{\vs_i\}$ to $\cT_i$.

Output $(\cT_1,\dots,\cT_n)$.

\paragraph{Remarks}
In the above algorithm we can terminate as soon as $\cT_i$ is found to be empty for any $i\ge 1$. It is easy to modify the above algorithm to only output all maximal tangles%
\footnote{Here, \emph{tangle} means an $ \cF $-tangle of $ \{s_1, \dots, s_i\} $ for some $ i \le n $, i.e.\ an element of $ \cT = \bigcup_i \cT_i $.
	It is \emph{maximal} if it is a $ \subseteq $-maximal element of $ \cT $.}
of $S$.

Finally, for purposes of parallelisation it is more efficient to implement the algorithm as a depth-first search rather than a breadth-first search over the (binary) tree of partial orientations.

\paragraph{Runtime}
The runtime of this algorithm is $O(|\cT|\cdot |S|^{m-1}\cdot f)$, where $\cT$ is the union of all $\cT_i$, $f$ is the time it takes to check whether a given subset of $S$ lies in $\cF$ and $ m $ is the maximum size of an element of $ \cF $. Note that the size of $\cT$ can be bounded by $|S|$ times the number of maximal tangles of $S$.
Indeed, in some applications reasonable bounds on this number of maximal tangles exist.
This runtime is calculated as follows: every tangle in $\cT$ is considered precisely once by the algorithm. If we consider $\tau\in \cT$, say, there is precisely one separation $s$ for which we check whether $\vs$ or $\sv$ can be added to $\tau$. For each of these checks, we consider every subset of size at most $m-1$ of $\tau$ and check whether this subset together with $\sv$ or $\vs$ forms an element of $\cF$. As there are at most $O(|S|^{m-1})$ such subsets, this whole procedure is $O(|S|^{m-1}\cdot f)$.

\section{Tangle-tree duality}

In~\cite{TangleTreeAbstract} it was shown that, under certain assumptions on $S$ and $\cF$, either $S$ has an $\cF$-tangle or there is an $S$-tree over $\cF$ which certifies that there can be no such $\cF$-tangle of~$S$. For a separation system satisfying the assumptions of the tangle-tree duality theorem given in~\cite{TangleTreeAbstract,AbstractTangles}, the algorithm described below will compute an $S$-tree over $\cF$ if there is no $\cF$-tangle of $S$, and otherwise find a consistent $\cF$-avoiding partial orientation which is a subset of every $\cF$-tangle of $S$, which can then be used as a starting point for the first algorithm which finds all tangles of $S$.

Our approach is to construct a list $L$ of all oriented separations that every $\cF$-tangle of $S$ has to contain by analysing the stars in $\cF$. As an example as well as the first step of the algorithm, consider a separation $ \vs\in\vS $ for which $\{\sv\}$ is a star in $\cF$. Then every $\cF$-avoiding orientation of $S$ and hence every $\cF$-tangle has to orient $s$ as $\vs$. Thus we can add all such separation $\vs$ to the list $L$ of all separations which every $\cF$-tangle is `forced' to contain. In every iteration of the algorithm we search $\cF$ for stars for which all but one of their elements lies in $L$. If $\sigma\in\cF$ is such a star and all elements of $\sigma$ other than some $\sv\in\sigma$ already lie in $L$, then similarly as above every $\cF$-tangle of $S$ must orient $s$ as $\vs$ in order to avoid $\sigma$; hence we can add every such $\vs$ to $L$.

When adding a separation $\vs$ to $L$ we keep track of the star in $\cF$ which caused the addition of $\vs$ to $L$. If at any point during our algorithm the list $L$ contains both orientation of some separation $s$, we can use this information to easily build an $S$-tree over $\cF$. Otherwise, if $L$ is always anti-symmetric, then eventually no star in $\cF$ will miss $L$ in exactly one element, at which point we output $L$. We will later show that the latter happens if and only if $S$ has an $\cF$-tangle, which then has to contain $L$.

\paragraph{Conditions}
Let $(\vU, \le, {}^*,\vee,\wedge)$ be a universe of separations containing a finite separation system $(\vS,\le,{}^*)$. Let $\cF\subseteq 2^{\vS}$ be a set of stars which is standard for $\vS$ and such that $\vS$ is $\cF$-separable.

\paragraph{Input}
The set $\cF$; oracle access to ${}^*$.


\paragraph{Output}
The information whether an $ \cF $-tangle of $ S $ or an $ S $-tree over $ \cF $ exists.
Either an $S$-tree over $\cF$, witnessing that there is no $\cF$-tangle of $S$, or if no $S$-tree over $\cF$ exists, a partial consistent $\cF$-avoiding orientation of $S$ such that every $\cF$-tangle of $S$ contains that partial orientation.


\paragraph{Algorithm}
Initialize two lists $L$ and $M$ as $\emptyset$.

Repeat the following procedure until it terminates:

Iterate over every star $\sigma$ in $\cF$. For every star $\sigma\in\cF$ check whether there is some element $\sv\in\sigma$ with $\vs\notin L$ for which all other elements of $\sigma$ are contained in $L$. If so, add $\vs$ to $L$ and add $(\vs,\sigma)$ to $M$. If, after this addition, both $\vs$ and $\sv$ lie in $L$, compute an $S$-tree over $\cF$ as described below and output that $S$-tree. If, on the other hand, $\cF$ contains no such star for which we add a separation to $L$, terminate and output $L$.

To construct an $S$-tree over $\cF$ in case $L$ contains both orientations of some separation~$s$, proceed as follows. For $\vr\in L$ let $\sigma_{\vr}$ denote the star with $(\vr,\sigma_{\vr})\in M$. Note that for every element $(\vr,\sigma_{\vr})$ of $M$ and every $\vt\in\sigma_{\vr}\sm\{\rv\}$, the separation $\vt$ lies in $L$ since otherwise $ (\vr, \sigma_{\vr}) $ wouldn't have been added to $ M $. In particular, $ \vt $ was added to $ L $ at an earlier step than $ \vr $, and there is a star $\sigma_{\vt}$ such that $ (\vt,\sigma_{\vt})\in M $. Observe further that if $ \sigma_{\vr}=\menge{\rv} $ then $ \menge{\rv}\in\cF $.

For the construction, start with the tree $T$ consisting of the vertices $\sigma_{\vs}$ and $\sigma_{\sv}$ joined by an edge labelled with $s$, with $\vs$ pointing towards $\sigma_{\sv}$. We now iterate the following until $T$ is an $S$-tree over $\cF$: for every leaf $\sigma_{\vr}$ of $T$ and every $\vt\in\sigma_{\vr}\sm\{\rv\}$, add the star $\sigma_{\vt}$ as a new vertex to $T$ and join it to $\sigma_{\vr}$ by an edge labelled with $t$, with $\vt$ pointing towards~$\sigma_{\vr}$. Since $ (\vt, \sigma_{\vt}) $ was added to $ M $ at an earlier step than $ (\vr, \sigma_{\vr}) $ was, this construction terminates.

\paragraph{Termination}
Since $S$ and $\cF$ are finite and the length of $L$ increases with every iteration over $\cF$ the algorithm eventually terminates.

\paragraph{Correctness}
It suffices to show that the algorithm outputs an $ S $-tree over $ \cF $ if and only if there is one. Clearly, the algorithm finding an $ S $-tree over $ \cF $ demonstrates the existence of an $ S $-tree over $ \cF $.

For the converse, suppose that $ (T,\alpha) $ is an $ S $-tree over $ \cF $, and let us show that the algorithm eventually finds and outputs an $ S $-tree over $ \cF $ (though not necessarily $ (T,\alpha) $ itself). For this it is enough to show that, for as long as $ L $ is anti-symmetric, at least one separation in the image of $ \alpha $ gets added to $ L $ during each iteration of the algorithm over $ \cF $: this would show that the algorithm cannot terminate with an anti-symmetric $ L $ and hence has to terminate on some $ S $-tree over $ \cF $.

So suppose that $ L $ is anti-symmetric and let $ P $ be a longest directed path in $ T $ without an edge whose inverse's image under $ \alpha $ lies in $ L $. Since $ L $ is anti-symmetric $ P $ is not a trivial path. Let $ v $ be the last vertex of $ P $ and $ w $ its neighbour on $ P $, and further let $ \sigma=\alpha(v) $ and $ \vs=\alpha(vw) $.
Then $ \sigma\in\cF $ will ensure that $ \vs $ gets added to $ L $ in this iteration: by definition of $ P $ we have $ \vs\notin L $, and $ L $ must contain every element other than $ \sv $ from $ \sigma $ by the maximality of $ P $.
Thus the algorithm can only terminate with an $ L $ which is anti-symmetric, meaning it outputs an $ S $-tree as shown above.

\paragraph{Runtime}
The runtime of this algorithm is $O(|\cF||S|^2)$: we iterate at most $ \abs{S} $ many times over $ \cF $, and every time we go through $ \cF $, we determine for every $ F\in\cF $ whether the size of $F \setminus L$ is $ 1 $. This can be done in $ O(\abs{S}) $ time, giving a total runtime of~$O(|\cF||S|^2)$. 

If the size of the elements of $ \cF $ is bounded the runtime of the algorithm decreases to $ O(|\cF||S|) $, because the check whether $ |F \setminus L| $ is 1 can then be implemented so as to run in constant time.

\paragraph{Remarks}
With a slight modification the above algorithm can be made to take $S$ and oracle access to $\cF$ as input rather than $\cF$ and oracle access to the involution of $S$. Here, by oracle access to $\cF$ we mean that we are able to check whether a given star in $S$ lies in $\cF$ or not.

Note further that the assumptions of the tangle-tree duality theorem, that $\cF$ is standard for $S$ and that $S$ is $\cF$-separable, were not used in the correctness proof that the algorithm finds an $S$-tree over $\cF$ whenever one exists. Indeed, these assumptions are only necessary to ensure, by using the tangle-tree duality theorem, that there is an $\cF$-tangle of $S$ in case the algorithm does not find an $S$-tree over $\cF$. 
In fact, the algorithm imitates a proof of the tangle-tree-duality theorem given by Bowler at a block seminar in Spr\"otze 2017. \cite{Sproetze}

If the size of the elements of $ \cF $ is bounded and $ \cF $ is given as a set rather than an oracle,
one can also construct a more elaborate data structure which stores for every $ \vs \in \vS $ the elements of $ \cF $ that contain it. With the help of this data structure the algorithm can be modified to run in $ O(|\cF| \log |\cF| + |S|) $ time.
However in practice one is rather unlikely to obtain $ \cF $ as a set: typically, $ \cF $ is defined as, say, the set of all stars whose interior has a certain size, in which case it is expensive to compute $ \cF $ as a set, but easy to check for any given subset $ \sigma $ of $ \vS $ whether $ \sigma $ lies in $ \cF $, i.e. to have oracle access to $ \cF $.

\section{Tree-of-Tangles  algorithm}

In this section we describe an algorithm for building a \emph{tree of tangles} as postulated by the tree-of-tangles theorem in \cite{ProfilesNew}.
The core feature of this algorithm is, that it does not require full information of the tangles and the separations system, but works by local improvements. It follows ideas from our `splinter theorem' in \cite{FiniteSplinters}.

Let $\vU=(U,{}^*,\le,\vee,\wedge,|\cdot|)$ be a universe with a submodular order-function.
For the sake of staying consistent with existing notation, let us assume that $ |\cdot| $ takes its values in $ \N_0 $, but do note that the algorithm works for any $ \R_+ $-valued order function.
Let $\vS\subseteq \vU$ be an abstract separation system, and let  $\cF \subseteq 2^{\vU}$. We assume that $ \cF $ is such that the $ \cF $-tangles of every $ S_k \subseteq U $ are robust profiles in $ \vU $.

We are interested in tangles of subsets of $ S $.
Let us call a consistent $ \cF $-avoiding orientation of all separations in $ S $ up to some order a \emph{tangle in $ S $}. A tangle \emph{of} $ S $ is one in $ S $ which orients \emph{every} separation of $ S $. A \emph{maximal} tangle in $ S $ is one which is $ \subseteq $-maximal among the tangles in $ S $.

Note that we are not making any assumptions about $ S $: $ S $ need not be of the form $ S_k $ for some $ k $; in fact, it need not even be structurally submodular. Thus some parts of the usual tangle theory and intuition do not apply to these tangles in $ S $: the tangles in $ S $ need not be, or extend to, tangles in $ U $. Let us call a tangle in $ S $ a {\em real} tangle in $ S $ if it is a subset of some tangle in $ U $, and a {\em fake} tangle in $ S $ otherwise.

Our aim is to find a nested set of separations distinguishing all maximal tangles in $ S $. As we made no assumptions about $ S $ at all, it might (and typically will be) impossible to find such a nested set inside of $ S $. Thus we need to define what it shall mean that a separation $ s\in U $ distinguishes a pair of tangles in $ S $, if those tangles need not contain some orientation of $ s $. Furthermore we might not be able to distinguish all maximal tangles in $ S $: by the Tree-of-Tangles-Theorem, we certainly will be able to distinguish all real tangles in $ S $ with a nested set if we are allowed to use separations from $ U $, but we might not be able to find a nested set which also distinguishes all fake tangles. Finally, $ U $ might be much bigger than $ S $, so we might not want to extend the real tangles of $ S $ to all of $ U $ in order to find distinguishers for them.

To distinguish the tangles of $ S $ we will extend them to also include some separations outside of $ S $. 
An \emph{extension of a tangle in $ S $}, or \emph{extended tangle} for short, is a tangle $ \tau' $ of some subset $ S' \subseteq U $ such that its intersection with $ \vS $ is a tangle $ \tau $ in $ S $, where the order of every separation in $ \tau' $ is at most as large as the order of some separation in $ \tau $.
A separation $ s \in U $ \emph{distinguishes} a pair of extended tangles $ \tau, \rho $ if it has an orientation $\vs$ with $ \vs \in \tau $ and $ \sv \in \rho $.

\paragraph{Overview}
The algorithm described below will find the following: a set $ \cT $ of extended tangles of $ S $ whose restrictions to $ \vS $ include all the real maximal tangles of $ S $, together with a nested set $ N\sub U $ which distinguishes all of $ \cT $.

Roughly speaking, the algorithm works as follows: we start with a list $ \cT $ of all maximal tangles of $ S $ and a set $ N $ of separations in $ S $ distinguishing all of $ \cT $. That is, $ N $ contains for every pair $ \tau,\rho $ of distinct tangles in $ \cT $ a separation $ \{\vs_{\tau,\rho},\vs_{\rho,\tau}\} $ with $ \vs_{\tau,\rho}\in\rho $ and $ \vs_{\rho,\tau}=\sv_{\tau,\rho}\in\tau $. Initially $ N $ is a subset of $ S $, but might not be nested. We will modify it to make it nested, which will make it necessary to replace some of its separations with elements of $ U \setminus S $.

For as long as $ N $ is not nested, we consider some pair of crossing separations $ s_{\tau, \rho} $ and $ s_{\phi,\psi} $ in $ N $ as well as a corresponding pair of pairs of tangles $ \{\tau,\rho\} $ and $ \{\phi,\psi\} $ in $ \cT $, and attempt to replace one of $ s_{\tau, \rho} $ or $ s_{\phi,\psi} $ in $ N $ with one of their four corners, without increasing the order of that separation $ s_{\tau, \rho} $ or $ s_{\phi,\psi} $. Here, `attempting to replace' $ s_{\tau,\rho} $ by the corner $ r $, say, of $ s_{\tau, \rho} $ and $ s_{\phi,\psi} $ means that we check whether $ r $ can be added in different orientations to the tangles $ \tau $ and $ \rho $ forming a pair of extended tangles, and if so, add those oriented separations to our two tangles, assigning $ s_{\tau,\rho} $ to be $ r $, adding $ r $ to the $ N $ and removing the previous $s_{\tau,\rho}$ from $ N $ (if it isn't still the $ s_{\tau',\rho'} $ for some other pair of $ \tau',\rho' $).

If this fails for all four corners, that is, if none of the four corners can be used to replace either $ s_{\tau, \rho} $ or $ s_{\phi,\psi} $, we shall be able to conclude that one of our four tangles involved, $ \tau $ say, is in fact a fake tangle. In that case we remove $ \tau $ from $ \cT $ and accordingly delete all separations from $ N $ which are not designated to distinguish a pair of tangles in the new, smaller, $ \cT $. By considering the crossing pairs of separations in a carefully chosen order we ensure that this replacement algorithm terminates. (We can delete fake tangles from $ \cT $ only finitely many times, but we need to take care not to run around in circles when we replace separations in $ N $.)

A difficulty here is that when we replace a separation of $ N $ which distinguishes a pair $ \tau,\rho $ of extended tangles in $ \cT $ by some $ r\in U $, we might be able to add $ \vr $ to $ \tau $ and $ \rv $ to $ \rho $, but it can happen that moreover $ \tau\cup\{\rv\} $ is an extended tangle of~$ S $. In that case we extend $ \tau $ by $ \vr $ but also add $ \tau\cup\{\rv\} $ to $ \cT $. This new tangle in $ \cT $ is distinguished from $ \tau $ by $ r $, and otherwise inherits its distinguishing separations from $ \tau $. Extra care will be needed to ensure that the algorithm terminates despite these splitting possibilities.

\paragraph{Input}
The separation system $ S $; the set of maximal tangles in $ S $; for every pair $ \tau,\rho $ of maximal tangles in $ S $ a separation $ \vs_{\tau,\rho} $ in $ S $ with $ \vs_{\tau,\rho}\in\rho $ and $ \vs_{\rho,\tau}=\sv_{\tau,\rho}\in\tau $; oracle access to ${}^*,\vee,\wedge, |\cdot|$ and $\cF$.

\paragraph{Output}
A set $ \cT $ of extended tangles whose restrictions to $ \vS $ include all real maximal tangles in $ S $; and for every pair $ \tau,\rho $ of tangles in $ \cT $ a separation $ s_{\tau,\rho}'\in U $ which distinguishes $ \tau $ and $ \rho $ and whose order is at most the order of $ s_{\tau,\rho} $, such that the set $ N $ of all these $ s_{\tau,\rho} $ is nested.

From now on, we shall treat the $ \vs_{\tau,\rho} $ as variables of the algorithm, so the $ \vs_{\tau,\rho} $ from the input will evolve throughout the algorithm and end up as the $ \vs'_{\tau,\rho} $ from the output.

\paragraph{Preliminaries}
Let $ \cT $ be a set of extended tangles. For $ \tau\in\cT $ and $ \vr\in\vU $ we say that we {\em can extend} $ \tau $ by $ \vr $, or that $ \vr $ {\em can be added} to $ \tau $, if $ \tau\cup\{\vr\} $ is again an extended tangle. If $ s $ is a separation distinguishing a pair of extended tangles $ \tau,\rho\in\cT $, we say that we {\em can replace} $ s $ (for $\tau$ and $\rho$) with some $ r\in U $ if $ \abs{r}\le\abs{s} $ and there are orientations $ \vr $ and $ \rv $ of $ r $ that can be added to $ \tau $ and $ \rho $, respectively. We say that replacing $ s $ for $\tau$ and $\rho$ with $ r $ {\em causes a split} if for one or both of $ \tau $ and $ \rho $ both $ \vr $ and $ \rv $ can be added to that tangle (yielding two new tangles).


\begin{LEM}\label{lem:findfake}
    Let $ (\tau, \rho)$ and $(\phi, \psi)$ be two pairs of extended tangles and let
    $ s $ and $ t $ be separations distinguishing $ \tau $ from $ \rho $
    and $ \phi $ from $ \psi $, respectively.
    Then one of the following holds:
    \begin{itemize}[nolistsep]
    \item[(1)] we can replace $ s $ with $ t $ or some corner of $ s $ and $ t $;
    \item[(2)] we can replace $ t $ with $ s $ or some corner of $ s $ and $ t $;
    \item[(3)] one of $ \tau,\rho,\phi $ and $ \psi $ is fake.
    \end{itemize}
\end{LEM}
\begin{proof}
    
By symmetry we may assume that $ \abs{s}\le\abs{t} $. If $ s $ has orientations $ \vs $ and $ \sv $ that can be added to $ \phi $ and $ \psi $, respectively, then $ t $ can be replaced with $ s $. Suppose this does not happen, so assume that some orientation of $ s $, say $ \sv $, can be added to neither of $ \phi $ and~$ \psi $. Since $ \abs{s}\le\abs{t} $ this means that for both of $ \phi $ and $ \psi $, either $ \vs $ can be added to that tangle, or we know that it must be fake. So suppose that $ \vs $ can be added to both $ \phi $ and~$ \psi $.

Consider the corner separations $ \vs\join\vt $ and $ \vs\join\tv $, where $ \vt\in\phi $ and $ \tv\in\psi $. If $ \abs{\vs\join\vt}\le\abs{t} $, then either $ \vs\join\vt $ can be added to $ \phi $ by the profile property, or we know that $ \phi $ must be fake; furthermore, we would know that $ (\vs\join\vt)^* $ can be added to $ \psi $ by consistency, or else $ \psi $ must be fake. Thus, if $ \abs{\vs\join\vt}\le\abs{\vt} $, either we discover that one of $ \phi $ or $ \psi $ is fake, or we can replace $ t $ with that corner. Similarly, the same statement holds for $ \vs\join\tv $.

So suppose that both of $ \vs\join\vt $ and $ \vs\join\tv $ have order strictly greater than $ \abs{t} $. Then the two opposing corners, $ \sv\join\vt $ and $ \sv\join\tv $, both have order strictly smaller than $ \abs{s} $. By symmetry we may assume that $ \vs\in\tau $ and $ \sv\in\rho $. Then by robustness one of those two corner separations can be added to $ \rho $ , or else we know that $ \rho $ is fake. In the first case, the inverse of that corner can be added to $ \tau $ by consistency, or else we know that $ \tau $ is fake. If neither of $ \rho $ and $ \tau $ is found to be fake in this way, then that corner can replace~$ s $.
\end{proof}

\paragraph{Algorithm}
Initialize $ \cT(0) $ as the set of all maximal tangles in $ S $, and initialize $ N(0) $ as the set of all $ \vs_{\tau,\rho} $ from the input. Fix an arbitrary enumeration of the (unordered) pairs of tangles in $ \cT(0) $; we say that the $n$-th pair in this enumeration has {\em index} $n$.
\\

We will iterate the following steps until we output the desired nested set. Let $ \ell $ denote the count of iterations. For the $\ell$th step, perform the following:\\

If $  N(\ell) $ is nested, output $ \cT(\ell) $ and $ N(\ell) $ and terminate the algorithm.
Otherwise for an integer $k$ let $ N_{<k}(\ell) $ denote the set of all $ \{\vs_{\tau,\rho},\vs_{\rho,\tau}\}\in N(\ell) $ where $ \{\tau,\rho\} $ has index~$ <k $. Consider as $ k = k(\ell) $ the largest integer such that $ N_{<k}(\ell) $ is nested.
Let $ \tau,\rho $ be the $ k $-th pair of tangles and $ s\coloneqq\{\vs_{\tau,\rho},\vs_{\rho,\tau}\} $.
For every $ t\coloneqq\{\vs_{\phi,\psi},\vs_{\psi,\phi}\}\in N_{<k}(\ell) $ which crosses $ s $ one of the following happens by \cref{lem:findfake}:
\begin{enumerate}[nolistsep]
	\item[(1)] we can replace $ s $ with $ t $ or some corner of $ s $ and $ t $;
	\item[(2)] we can replace $ t $ with $ s $ or some corner of $ s $ and $ t $;
	\item[(3)] we find that one of $ \tau,\rho,\phi $ and $ \psi $ is fake.
\end{enumerate}
If (1) occurs for {\em some} $ t\in N_{<k}(\ell) $ which crosses $ s $, replace $ s $ with $ t $ or that corner of $ s $ and $ t $ to obtain $ N(\ell + 1) $ and $ \cT(\ell+1) $ as follows: let $ \cT(\ell+1) $ be the set $ \cT(\ell)\sm\menge{\tau,\rho} $ together with all elements of $ E\coloneqq\menge{\tau\cup\menge{\vr},\tau\cup\menge{\rv},\rho\cup\menge{\vr},\rho\cup\menge{\rv}} $ that are again extended tangles. In other words, $ \cT(\ell+1) $ is obtained from $ \cT(\ell) $ by performing all possible extensions of $ \tau $ and $ \rho $ by $ r $. Since $ s $ can be replaced by $ r $, for some orientation $ \vr $ of $ r $ we have that both $ \tau\cup\menge{\vr} $ and $ \rho\cup\menge{\rv} $ are extended tangles and hence in $ \cT(\ell+1) $. The replacement of $ s $ with $ r $ causes a split if and only if at least one of the other two elements of $ E $ is also an extended tangle. We enumerate the (unordered) pairs of tangles in $ \cT(\ell+1) $ by following the enumeration of the pairs in $ \cT(\ell) $, substituting $ \tau $ and $ \rho $ in those pairs with $ \tau\cup\menge{\vr} $ and $ \rho\cup\menge{\rv} $, respectively, and (if there was a split) appending all pairs containing $ \tau\cup\menge{\rv} $ or $ \rho\cup\menge{\vr} $ at the end of the enumeration in an arbitrary order. To obtain $ N(\ell+1) $, for a pair $ \xi,\zeta $ of extended tangles in $ \cT(\ell+1) $ we take the following as distinguishing separation $ s_{\xi,\zeta}\in N(\ell+1) $ : if both of $ \xi $ and $ \zeta $ lie in $ E $ we set $ s_{\xi,\zeta}\coloneqq r $ if $ r $ distinguishes $ \xi $ and $ \zeta $ and otherwise $ s_{\xi,\zeta}\coloneqq s $; if neither of $ \xi $ and $ \zeta $ lies in $ E $, we take the $ s_{\xi,\zeta} $ from $ N(\ell) $; and finally, if exactly one of $ \xi $ and $ \zeta $ lies in $ E $, say $ \zeta\in E $, we take as $ s_{\xi,\zeta} $ the $ s_{\xi,\zeta\sm r} $ from~$ N(\ell) $. We then continue with the next step~$ \ell+1 $.


Else, if (2) occurs for {\em all} $ t\in N_{<k}(\ell) $ which cross $ s $, we obtain $ N(\ell + 1) $ and $ \cT(\ell+1) $ by performing those replacements one after the other as above, considering those $ t $ crossing $ s $ in increasing order of the index of the pairs of tangles. If, after some of these replacements have taken place, we find that we can no longer replace the next $ t $ with $ s $ or one of its corners with $ s $, we can conclude that one of the tangles corresponding to that $ t $ or $ s $ is fake: we can apply \cref{lem:findfake} to the extended tangles resulting from the replacements so far. Then (1) cannot apply to these extended tangles since (1) did not apply to the original extended tangles. In the case where we find a fake tangle, we omit this fake tangle from $ \cT(\ell+1) $ and its associated separations from $ N(\ell + 1) $ and continue with the next step, $ \ell + 1 $.
Otherwise, if all those replacements were successful, $ s=s_{\tau,\rho} $ will be nested with $ N_{<k}(\ell + 1) $, but $ N_{<k}(\ell + 1) $ itself might not be nested anymore; we also continue with the next step $ \ell + 1 $. 

Else (3) occurs for some $ t\in N_{<k}(\ell) $ crossing $ s $. We then remove the fake tangle from $ \cT(\ell) $ to obtain $ \cT(\ell + 1) $ and remove all separations associated with that tangle from $ N(\ell) $ to obtain $ N(\ell + 1) $.
Then continue with step $ \ell + 1 $.

\paragraph{Termination and Correctness}\label{par:tot_termination_and_correctness}
We define a quasi-order on the set of all sets $ \cT $ of extended tangles by letting $ \cT\prec\cT' $ if for some $ p $ the set $ \cT $ contains strictly fewer $ p $-element tangles than $ \cT' $, and for all $ q<p $ the number of $ q $-element tangles in $ \cT $ and $ \cT' $ is the same. Observe that throughout the algorithm above, $ \cT $ never increases with respect to this quasi-order, i.e.\ $\cT(\ell+1) \preceq \cT(\ell)$ for all $\ell$. In fact $\cT(\ell+1)$ and $\cT(\ell)$ are equal as sets whenever they are equivalent in the quasi-order.

\begin{LEM}\label{lem:induction_tot}
If $\ell$ is such that in the $\ell$-th step of the algorithm $\cT$ is minimal with respect to the quasi-ordering defined above, and $k$ is a natural number, then there is a $j$ for which the set $N_{<k}(\ell+j)$ is nested. Moreover, every separation which was nested with $N_{<k}(\ell)$ is still nested with $N_{<k}(\ell+j)$.
\end{LEM}
\begin{proof}
We proceed by induction on $k$. The assertion clearly holds for $ k=1 $ with $j=0$.

So suppose that $ k>1 $ and that the above assertion holds for $k-1 $. Let $j_1$ be such that $N_{<(k-1)}(\ell+j_1)$ is a nested set that is nested with every separation which was nested with $N_{<(k-1)}(\ell)$ and, subject to this, such that the $k$-th separation $s$ crosses as few elements of $N_{<(k-1)}(\ell+j_1)$ as possible. If this separation is nested with $N_{<(k-1)}(\ell+j_1)$, we are done.
Otherwise, (1) happens for all $ t\in N_{<k}(\ell+j_1) $ which cross $ s $. Thus, we replace all these $t$ with some corner of $t$ and $s$. After this replacement, $N_{<(k-1)}(\ell+j_1+1) $ is now nested with $s$. Thus, by applying the induction hypothesis again there is some $j_2$ such that the set $N_{<(k-1)}(\ell+j_1+j_2)$ is nested again and, by the moreover part of the statement, is also nested with $s$. Additionally, every separation which was nested with $N_{<k}(\ell)$ is also nested with $N_{<k}(\ell+j_1+j_2)$ by the fish lemma.
\end{proof}

Since the sequence $ (\cT(\ell))_{\ell\ge 0} $ is decreasing in the quasi-order, and there are only finitely many extended tangles, there is an $ \ell_0 $ such that $ \cT(\ell) $ and $ \cT(\ell_0) $ are equivalent in the quasi-order for all $ \ell\ge\ell_0 $. Then $ \cT(\ell) $ and $ \cT(\ell_0) $ are also equal as sets for all $ \ell\ge\ell_0 $. Let~$k$ be the number of pairs of tangles from $\cT(\ell_0)$. By Lemma \ref{lem:induction_tot}, for some $\ell_1\ge\ell_0$ the set $N_{<(k+1)}(\ell_1)$ will be nested. Then $ N=N_{<(k+1)}(\ell_1) $ is the desired nested set, and the algorithm indeed terminates in step $ \ell_1 $.


\paragraph{Runtime}
Let $l$ denote the number of times that a replacement causes a split throughout the algorithm. We will later discuss a possible {\em a priori}  bound on $l$. For our calculation of the worst-case runtime we may assume that we never delete tangles that are fake, since deleting fake tangles from $\cT$ only reduces the remaining runtime. Furthermore, for calculating the worst-case runtime, it does not make a difference at which point during the algorithm the replacements causing splits take place: thus, for simplicity, we will assume that we start with a list of $\abs{\cT}+l$ many tangles, that is, with $\binom{\abs{\cT}+l}{2}$ many pairs of tangles, and that from there on no replacement of a separation causes any further splits.

We shall compute the worst-case runtime of our algorithm recursively. Let $r_m$ denote the worst-case runtime of the algorithm for a list of $r_m$ pairs of tangles. For a list of $m+1$ pairs of tangles the algorithm first runs on the sub-list of the first $m$ pairs of tangles, computing a nested set $N_{<m+1}$ which distinguishes the first $m$ pairs of tangles. From there the algorithm considers the separation $s:=s_{\tau,\rho}$, where $\tau,\rho$ is the $(m+1)$-st pair $\tau,\rho$ of tangles, and iterates over $N_{<m+1}$ to check which of (1), (2) or (3) occurs for the $t\in N_{<m+1}$ which cross $s$. By our assumption above (3) does not occur. Let $c$ denote the longest possible time it takes to check for a single $t\in N_{<m+1}$ crossing $s$ which of (1), (2) and (3) occurs. $c$ depends on the time it takes to check whether a separation can be added to a given tangle, which in turn is a function of both the length of that tangle as well as the size of the elements in $\cF$; again, we postpone our estimation of $c$ and will treat $c$ as a constant.

If $k$ is the number of times that (1) happens for $s$ and some $t\in N_{<(m+1)}$, then the algorithm performs at most $k$ replacements of $s$, each time going through the entire $m$-element list $N_{<(m+1)}$ and checking each element of $N_{<(m+1)}$ in time $c$. Following the $(k+1)$-st iteration, in which no $t\in N_{<(m+1)}$ with (1) is found, the algorithm performs the replacements of all $t\in N_{<(m+1)}$ with (2) simultaneously. The subset of those newly replaced $t\in N_{<(m+1)}$ might not be nested any more, and it takes the algorithm $r_{m-k}$ time to make it nested again.\footnote{
	Note that, after these replacements, the only separations in $N_{<(m+1)}$ that the replacement of some $t$ for which (2) happened can cross are replacements of other $t'\in N_{<(m+1)}$ for which (2) happened.
}
Thus we have
\[ r_{m+1}=O\braces{r_m+c\cdot(k+1)\cdot m+r_{m-k}}. \]
We will show by induction that
\[ r_{m+1}=O\braces{ c\braces{2^{m+1} - (m+2)} }. \]
The induction start for $r_2$ is easy. Suppose now that the above equation holds for all smaller values of $m$ and let us derive it for $r_{m+1}$. Then for $m\ge 2$ the function $ {k\mapsto c\cdot(k+1)\cdot m+r_{m-k}} $ attains its maximum at $k=0$, showing that for the worst-case runtime occurrences of (2) are worse than those of (1). With this we have
\begin{align*}
r_{m+1}&=O\braces{r_m+c\cdot(k+1)\cdot m+r_{m-k}}\\
&=O\braces{ 2\cdot c\braces{2^{m} - (m+1)} + c\cdot m}\\
&=O\braces{ c\braces{2^{m+1} - (m+2)} },
\end{align*}
concluding the induction.

Therefore the total runtime is
\[ O\braces{c\cdot2^{\binom{\abs{\cT}+l}{2}}}. \]
Recall that $c$ depends on the structure of $\cF$; in particular, if the size of the elements of $\cF$ is unbounded, there might not be a polynomial bound on $c$ in terms of the maximal length of the tangles involved. Let us now discuss possible bounds on $l$, the number of times that a replacement may cause a split.

If $U'\subseteq U$ is the smallest sub-universe of $U$ which contains $S$, the only {\em a priori} bound on $l$ we can get is $l\le|\cT|\cdot 2^{|U'|}$, since every tangle in $\cT$ might be extendable by every orientation of the rest of $U'$. This results in an abysmal runtime of
\[ O\braces{ c\cdot 2^{\binom{|\cT|+|\cT|\cdot2^{|U'|}}{2} } }. \]

However, in practice the runtime is not that bad: first of all, replacements only rarely cause splits. In particular, if at the start of the algorithm every tangle in $\cT$ orients all of $S$, then initially no replacement can cause a split since the algorithm is initiated with ${N\sub S}$.

Moreover, in practice the maximal length of a tangle in $\cT$ does not increase much throughout the algorithm, and hence $O(\abs{S})$ is a reasonable bound on this length in practice, resulting in a polynomial bound on $c$ if the elements of $\cF$ are of bounded size.

Out of the three cases (1), (2) and (3) which can occur during the algorithm, any instance of (3) happening reduces the runtime significantly: not only does $\cT$ get smaller, but deleting a fake tangle also reduces the possibilities for causing splits in the future.

Finally, and most importantly, the theoretic calculations above assume that two separations cross whenever we did not make them nested before. However, in practice it happens quite often that two separations of $N$ are already nested and no replacement is necessary, for instance simply due to both being the same separation. 



\bibliographystyle{abbrv}
\bibliography{collective.bib}

\end{document}